\documentclass[a4paper,12pt]{amsart}

\usepackage[english]{babel}
\usepackage{tikz}
\usetikzlibrary{positioning,arrows,cd}

\usepackage{amsmath,amssymb,amsthm}

\usepackage{enumitem}

\usepackage{verbatim}
\newtheorem{theorem}{Theorem}[section]
\newtheorem{thmx}{Theorem}
\newtheorem{corx}[thmx]{Corollary}
\newtheorem{proposition}[theorem]{Proposition}
\newtheorem{lemma}[theorem]{Lemma}
\newtheorem{corollary}[theorem]{Corollary}

\newtheorem{question}[theorem]{Question}

\theoremstyle{definition}
\newtheorem{definition}[theorem]{Definition}

\theoremstyle{remark}
				
\newtheorem{remark}[theorem]{Remark}

\newcommand{\Aut}{\operatorname{Aut}}

\newcommand{\K}{{\Bbbk}}

\newcommand{\G}{\mathcal{G}}
\newcommand{\A}{\mathcal{A}}

\newcommand{\B}{{\mathcal B}}

\usepackage[pdfpagelabels]{hyperref}

\usepackage{colortbl}
\DefineNamedColor{named}{RoyalBlue}     {cmyk}{1,0.50,0,0}
\DefineNamedColor{named}{BrickRed}      {cmyk}{0,0.89,0.94,0.28}

\begin{document}

\title[Permutation representations and evolution algebras]{
Permutation representations and automorphisms of evolution algebras
}

\author{Cristina Costoya}
%    Address of record for the research reported here
\address{Departamento de Matemáticas (USC), and Centro de Investigación y Tecnología Matemática de Galicia (CITMAga)}
\email{cristina.costoya@usc.es}
%    \thanks will become a 1st page footnote.
%\thanks{The first author was partially supported by the Spanish Ministerio de Ciencia e Innovaci\'on, project PID2020-115155GB-I00}

%    Information for first author
\author{Pedro Mayorga}
%    Address of record for the research reported here
\address{Departamento de \'Algebra, Geometr{\'\i}a y Topolog{\'\i}a, Universidad de M{\'a}laga, 29071-M{\'a}laga, Spain}
\email{pedromayped@uma.es}
%    \thanks will become a 1st page footnote.

%    Information for  author
\author{Antonio Viruel}
\address{Departamento de \'Algebra, Geometr{\'\i}a y Topolog{\'\i}a, Universidad de M{\'a}laga, 29071-M{\'a}laga, Spain}
\email{viruel@uma.es}
%\thanks{The third author was partially supported by the Spanish Ministerio de Ciencia e Innovaci\'on, project PID2020-118753GB-I00, and Andalusian Consejer{\'\i}a de Universidad, Investigaci\'on e Innovaci\'on, project PROYEXCEL-00827}


\begin{abstract}

We prove that the natural permutation representation of highly transitive finite groups cannot be realized as the full automorphism group of an idempotent, finite-dimensional evolution algebra acting on the set of lines spanned by its natural elements. Specifically, for any sufficiently large integer $n$ and $k \geq 4$, there does not exist an idempotent evolution algebra $X$ of dimension $n$ such that $\operatorname{Aut}(X)$ is isomorphic to a proper $k$-transitive subgroup of $S_n$.

Nevertheless, we show that for any finite group $G$, any permutation representation $\xi \colon G \to S_n$, and any field $\Bbbk$, there exists an idempotent, finite-dimensional evolution $\Bbbk$-algebra $X$ such that $\operatorname{Aut}(X) \cong G$, and the induced representation of $\operatorname{Aut}(X)$ on the natural idempotents of $X$ is equivalent to $\xi$.

\end{abstract}

\maketitle

%%%%%%%%%%%%%% INTRODUCTION %%%%%%%%%%%%%%%%%%%
\section{Introduction}

An $n$-dimensional evolution algebra $X$, defined over an arbitrary field $\Bbbk$, is a not necessarily associative and commutative algebra equipped with a natural basis $B = \{b_1, \ldots, b_n\}$ that satisfies the key property: $b_i b_j = 0$ whenever $i \neq j$. A notable feature of these algebras is that once the natural basis is fixed,  the products $b_i^2$ for $i = 1, \ldots, n$ are encoded in a matrix $M_B(X)$, known as the structure matrix, that fully determines the structure of the evolution algebra.

First introduced by Tian~\cite{Tian-Vojtechovsky,Tian}, evolution algebras serve as a framework to model self-reproduction in non-Mendelian genetics, building upon the concept of genetic algebras originally defined by Etherington~\cite{Etherington}. Since the publication of Tian's seminal works, they have attracted considerable attention, leading to significant progress in understanding their structure~\cite{Cabrera-Siles-Velasco, tensor_product, natural_element2, Camacho-Gomez, Casas-Ladra} and, in their classification for low-dimensional cases~\cite{CABRERACASADO201768, casado2022simple, 2-dim-perfect, CALDERONMARTIN2024107594, ELDUQUE201611, Sriwongsa-Zou}.

A central result in the study of automorphism groups of evolution algebras, with implications for their classification, is due to Elduque and Labra~\cite{Elduque-Labra-2019}. They showed that for an $n$-dimensional idempotent evolution algebra $X$ (i.e., $X^2 = X$), there exists an exact sequence:
\begin{equation}\label{eq:intro}
1 \to D \to \Aut(X) \xrightarrow{\rho_X} S_n,
\end{equation}
where $ \Aut(X) $ denotes the automorphism group of the algebra, $ S_n $ is the symmetric group on $ n $ elements, and $ D $ is a finite abelian group related to unity roots in $ \Bbbk $.
Hence, a direct consequence of Equation \eqref{eq:intro} is that for any finite-dimensional idempotent evolution $ \Bbbk $-algebra $ X $, $ \Aut(X) $ is a finite group.

Conversely, it is known that any finite group is the automorphism group of an idempotent evolution algebra. More precisely, for any finite group $ G $, and any field $ \Bbbk $, there exists a finite-dimensional idempotent evolution $ \Bbbk $-algebra $ X $ such that $ \Aut(X) \cong G $ \cite[Theorem 1.1]{CLTV}. This result was independently proved when $ \operatorname{char}(\Bbbk) = 0 $ \cite[Theorem 3.1]{Sriwongsa-Zou}. Moreover, when the units of $ \Bbbk $ satisfy $ |\Bbbk^*| \geq 2|G| $, then evolution algebra $ X $ can be chosen to be simple \cite[Theorem A]{CMTV}.

In the light of Equation \eqref{eq:intro} and the preceding result  \cite[Theorem 1.1]{CLTV},  it is natural to ask whether permutation representations of finite groups can be realized within this framework, so we formulate the following problem:
\begin{question}\label{Quest:basica}
   Let  $\Bbbk$  be a field and let $G$ be a finite group. For any permutation representation  $\xi\colon G\to S_n$  with abelian kernel, does there exist a finite-dimensional idempotent evolution $\Bbbk$-algebra $X$ such that $\Aut(X)\cong G$, and the induced permutation representation $\rho_X: \Aut(X) \rightarrow S_n$  in Equation \eqref{eq:intro} is equivalent to $\xi$?
\end{question}
Our first result shows that the answer to Question \ref{Quest:basica} may be negative. Specifically,  a high transitivity of the subgroup $\rho_X\big(\Aut(X)\big)\leq S_n$ (see Definition \ref{def:transitive}) imposes restrictions on the behavior of $\rho_X$:

\begin{thmx}\label{thmx:main_A}
   Let $\Bbbk$ be a field and let $X$ be an $n$-dimensional idempotent evolution $\Bbbk$-algebra.
    \begin{enumerate}[label={\rm (\roman{*})}]   
        \item\label{thmx:main_A.1} For $n\geq 3$, if $\rho_X\big(\Aut(X)\big)\leq S_n$ is $2$-transitive then $\rho_X$ is faithful, that is $\Aut(X)\cong \rho_X\big(\Aut(X)\big).$ 
        
        \item\label{thmx:main_A.2} For $n\geq 5$, if $\rho\big(\Aut(X)\big)\leq S_n$ is $4$-transitive then $\rho_X$ is faithful and full, that is $\rho_X$ induces an isomorphism $\Aut(X)\cong S_n.$\
    \end{enumerate}
\end{thmx}

In \cite{Sriwongsa-Zou},  it is claimed that $n$-dimensional idempotent evolution $\Bbbk$-algebras can be classified based on the isomorphism type of their automorphism group. 
Our result contributes to this classification problem by providing the following characterization, which complements \cite[Theorem 3.2]{Sriwongsa-Zou}:

\begin{corx}\label{corx:trans}
 Let  $\Bbbk$  be a field and let $G\leq S_n$ be a $k$-transitive group, $k\geq 4$. Then, $G$ is the automorphism group of an  $n$-dimensional idempotent evolution $\Bbbk$-algebra if and only if $G=S_n$. 
\end{corx}

Permutation representations of the full automorphism group of a given finite-dimensional idempotent evolution $\Bbbk$-algebra $X$ can be constructed in ways other than described in Equation \eqref{eq:intro}. Specifically, let $\widetilde{\B}_X$ be the set of all the idempotent natural elements of $X$ (see Definition \ref{def:natural_elements}).  Since $\Aut(X)$ acts naturally on $\widetilde{\B}_X$, it induces a permutation representation
\begin{equation}\label{eq:intro2}
    \widetilde{\rho}_X: \Aut(X) \rightarrow S_m
\end{equation}
where $m=|\widetilde{B}_X|$. Hence,  similarly to Question \ref{Quest:basica}, we raise the following problem:

\begin{question}\label{Quest:basica2}
   Let  $\Bbbk$ be a field and let $G$ be a finite group. For any permutation representation $\xi\colon G\to S_m$ , does there exist a finite-dimensional idempotent evolution $\Bbbk$-algebra $X$  such that $\Aut(X)\cong G$, and the induced permutation representation $\widetilde{\rho}_X: Aut(X) \rightarrow S_m$  in Equation \eqref{eq:intro2} is equivalent to $\xi$?
\end{question}

Unlike Question \ref{Quest:basica},  we solve Question \ref{Quest:basica2} affirmatively: 

\begin{thmx}\label{thmx:main_C}
Let $\K$ be a field and $G$ be a finite group. For any  permutation representation $\xi\colon G\to S_m$,  there exist infinitely many non isomorphic finite-dimensional idempotent evolution $\K$-algebras $X$ such that $\Aut(X)\cong G$, and the induced permutation representation $\widetilde{\rho}_X$  in Equation \eqref{eq:intro2} is equivalent to $\xi$.
\end{thmx}

\subsection*{Organization of this paper}
We start by giving an introduction to evolution algebras (Section \ref{sec:evol_alg}) and multiply-transitive groups (Section \ref{sec:trans_groups}). 
In Section \ref{sec:Proof_A} we prove Theorem \ref{thmx:main_A} and Corollary \ref{corx:trans}. The final section is dedicated to proving  Theorem \ref{thmx:main_C}, presented there as Theorem \ref{thm:repres}.

\subsection*{Funding}
This work was partially supported by MCIN/AEI/10.13039/501100011033 [PID2020-115155GB-I00 to C.C.,  PID2020-118753GB-I00 to P.M.\ and A.V., and PID2023-149804NB-I00 to C.C. and A.V.]. The third author was also partially supported by PAIDI 2020 (Andalusia) grant PROYEXCEL-00827.

%%%%%%%%%%%%%%%%%%%%%%%%%%%%
%%%%%%%%%%% SECTION 2 %%%%%%%%%%%
%%%%%%%%%%%%%%%%%%%%%%%%%%%%
\section{Basics of evolution algebras}\label{sec:evol_alg}

In this section we review the necessary background on evolution algebras that  is needed throughout this article. We refer  to the monograph by Tian \cite{Tian} for a thorough description on evolution algebras.
However, most of our arguments are based on the results presented in \cite{Elduque-Labra-2015, Elduque-Labra-2019}.

All the algebras considered in this work will be defined over a  field $\K$ of arbitrary characteristic and have finite dimension. A $\K$-algebra is just a vector space $X$ endowed with a $\K$-bilinear map (multiplication) $X \times X \rightarrow X$ given by  $(x, y) \mapsto xy$. 

\begin{definition}\label{def:evol_alg} An $n$-dimensional evolution algebra is a $\K$-algebra $X$, provided with a basis $\B=\{b_1, \ldots, b_n\}$ satisfying $b_ib_j=0$ for any $1 \leq i \neq j \leq n$. Such a basis is referred to as a natural basis.
\end{definition}

Evolution algebras are non-associative and commutative \cite[Corollary 1]{Tian}. Although natural basis are not unique (see \cite[Examples 2.4 and 2.5]{Elduque-Labra-2015}),  once a natural basis is chosen, the squares of its elements fully determine the structure of the algebra. This fact motivates the following definition:

\begin{definition}\label{def:struc_matrix} Let $X$ be an evolution $\K$-algebra with natural basis $\B=\{b_1, \ldots, b_n\}.$ The structure matrix of $X$ relative to $\B$ is defined as $M_\B  := (\mu_{ij})$ where $b_i^2=\sum_{j=1}^n\mu_{ij}b_j$, for $i=1, \ldots, n$. The coefficients $\mu_{ij} \in \K, \,  i, j = 1, \ldots, n,$ are referred to as structure constants. 
\end{definition}
It is worth noting that some authors adopt a different notation for the structure constants, resulting in a structure matrix that is the transpose of ours.
We will focus on a particularly special class of evolution algebras:
\begin{definition}\label{def:idempotent_alg}
An evolution $\K$-algebra $X$ is idempotent (or regular, or perfect) 
if  $X^2=X$.
\end{definition}
A classical argument in linear algebra shows that an evolution algebra is idempotent if and only if its structure matrix, with respect to any natural basis, is invertible \cite[Remark 4.2]{Elduque-Labra-2015}. Idempotent evolution algebras are particularly convenient since they possess a unique natural basis, up to scaling and reordering of its elements:
\begin{theorem}[{\cite[Theorem~4.4]{Elduque-Labra-2015}}]\label{thm:unique_basis}
Let $X$ be an $n$-dimensional idempotent evolution $\K$-algebra and let $\B=\{b_1, \ldots, b_n\}$  and $\B'=\{b'_1, \ldots, b'_n\}$ be natural bases of $X$. Then,  there exist nonzero  scalars $\lambda_i \in \K$, and a permutation $\sigma \in S_n$, such that for $i=1, \ldots, n$, $b'_i=\lambda_i b_{\sigma(i)}$.\end{theorem}

This leads to a characterization of the automorphisms of finite-dimensional idempotent evolution algebras:

\begin{corollary}\label{cor:cor} 
Let $X$ be an $n$-dimensional  idempotent evolution $\K$-algebra, let $\B=\{b_1, \ldots, b_n\}$  be a natural basis of $X$ and $\varphi \in\Aut(X)$. Then, there exist nonzero scalars $\lambda_i \in \K$ and a permutation $\sigma \in S_n$ such that $\varphi(b_i)=\lambda_i b_{\sigma(i)}$,  $i=1, \ldots, n$.
\end{corollary}

\begin{remark}\label{rmk:equivalent_representations}
    If $X$ is an idempotent evolution $\K$-algebra with natural basis $\B=\{b_1,\ldots,b_n\}$, Corollary \ref{cor:cor} allows us to map each $\varphi \in \Aut(X)$ to a permutation $\rho_{(X, \B)}(\varphi) := \sigma \in S_n$. This map is a group homomorphism and corresponds to the permutation representation of $\Aut(X)$ in Equation \eqref{eq:intro}. Notice that $\rho_{(X, \B)} $ depends on the chosen natural basis $\B$, but if $\B'$ is a different basis, there exists a permutation $\tau \in S_n$ such that $\rho_{(X, \B')}(\varphi) = \tau^{-1} \cdot \rho_{(X, \B)}(\varphi) \cdot \tau$. Therefore, the representations $\rho_{(X, \B)}$ and $\rho_{(X, \B')}$ are equivalent  (see \cite[p.\ 35]{Robinson}) and we adopt the simplified notation $\rho_X$ for the representation, ignoring the choice of basis.
\end{remark}

We now proceed to define the permutation representation $\widetilde{\rho}_X$ in Equation \eqref{eq:intro2}. We first define natural elements in an evolution algebra. 

\begin{definition}[{\cite[Definitions 2.2]{natural_element}}]\label{def:natural_elements}
Let $X$ be an evolution $\K$-algebra. A set 
$\A=\{x_1,\ldots, x_r\}\subset X$ is called an extending natural family if there exists a natural basis $\B$  of $X$ such that $\A\subset\B$.  We say that $x \in X$ is a natural element if $\{x\}$ is an extending natural
family. 
\end{definition}

A very special example of an extending natural family is the set of all idempotent natural elements in an idempotent evolution algebra:

\begin{proposition}\label{prop:idemp_natural}
Let $X$ be an $n$-dimensional idempotent evolution $\K$-algebra. The set of all idempotent natural elements in $X$, denoted $\widetilde{\B}_X$, forms an extending natural family. Specifically, $\widetilde{\B}_X$ is a finite set with size $m := |\widetilde{\B}_X| \leq n$.
\end{proposition}

\begin{proof}
Let $\B' = \{b'_1, \ldots, b'_n\}$ be a natural basis of $X$ with structure constants $(\mu'_{ij})$. We focus on the elements of $\B'$ that satisfy ${b'_i}^2 = \mu'_{ii} b'_i$ with $\mu'_{ii} \neq 0$. After reordering, we can assume that the first $m$ elements, $b'_1, \ldots, b'_m$, are the only ones in $\B'$ that meet this condition. If no such elements exist, we set $m = 0$.
Let
$$
b_i = 
\begin{cases} 
(\mu'_{ii})^{-1} b'_i & \text{for } i \leq m, \\
b'_i & \text{for } i > m.
\end{cases}
$$
This defines $ \B = \{b_1, \dots, b_n\} $ as a natural basis of $ X $ with structure constants $ (\mu_{ij}) $. For $ i \leq m $, 
$$ b_i^2 = (\mu'_{ii})^{-2} {b'_i}^2 = (\mu'_{ii})^{-1} b'_i = b_i, $$
which shows that $ \{b_1, \dots, b_m\} \subset \widetilde{\B}_X $. Additionally, by construction, we have that $ \mu_{ii} = 1 $, and $ \mu_{ij} = 0 $ for $ i \neq j $ if and only if $ i \leq m $.

We conclude by proving that $ \widetilde{\B}_X \subset \{b_1, \dots, b_m\} $. Indeed, for any $ x \in \widetilde{\B}_X $, since $ x $ is a natural element, there exists a natural basis $ \A $ such that $ x \in \A \subset X $. By Theorem \ref{thm:unique_basis}, there is a nonzero scalar $ \lambda_0 \in \K $ and an index $ i_0 \in \{1, \dots, n\} $ such that $ x = \lambda_0 b_{i_0} $. Since $ x = x^2 $, we have:
$$ \lambda_0 b_{i_0} = (\lambda_0 b_{i_0})^2 = \lambda_0^2 \sum_{j=1}^n \mu_{i_0j} b_j. $$
Comparing coefficients, we find that $ \mu_{i_0j} = 0 $ for $ j \neq i_0 $ and $ \mu_{i_0i_0} \neq 0 $. This implies that $ i_0 \leq m $, and $ \mu_{i_0i_0} = 1 $. Therefore, $ \lambda_0 = 1 $, and consequently, $ x = b_{i_0} \in \{b_1, \dots, b_m\} $.
\end{proof}

We can now describe the permutation representation $\widetilde{\rho}_X$ from Equation \eqref{eq:intro2}.
Let $X$ be an $n$-dimensional idempotent evolution $\K$-algebra, $\widetilde{\B}_X = \{b_1, \ldots, b_m\}$ the set of all idempotent natural elements in $X$, and $\B = \{b_1, \ldots, b_n\}$ the natural basis with $\widetilde{\B}_X \subset \B$, as described in the proof of Proposition \ref{prop:idemp_natural}. 

We considerer  the permutation representation $\rho_X \colon \Aut(X) \to S_n$ constructed in Remark \ref{rmk:equivalent_representations}, which appears in Equation \eqref{eq:intro}. 
Given $\varphi \in \Aut(X)$, note that $\varphi$ maps idempotent elements to idempotent elements. Thus, if $i \leq m$ (so that $b_i \in \widetilde{\B}_X$ and $b_i$ is idempotent), we have $\varphi(b_i) = b_{\rho_X(\varphi)(i)}$, where $\rho_X(\varphi)(i) \leq m$. 
This defines a new permutation representation
$$
\widetilde{\rho}_X \colon \Aut(X) \to S_m,
$$
induced by the action of $\Aut(X)$ on $\widetilde{\B}_X$, which is given by
$$
\widetilde{\rho}_X(\varphi)(i) = \rho_X(\varphi)(i) \quad \text{for } i = 1, \ldots, m.
$$
In other words, the permutation representation $\rho_X \colon \Aut(X) \to S_n$ can be decomposed as:
$$
\begin{tikzcd}
\Aut(X)  \ar[dr, "{\widetilde{\rho}_X\oplus \alpha}" ', bend right=20] \ar[rr, "{\rho}_X"]& &  S_n \\
& S_m\times \operatorname{Sym}(\{m+1,\ldots, n\}) \ar[ur, hook, bend right=20] &
\end{tikzcd}
$$

\begin{remark}\label{rmk:equivalent_idempotent_representation}
As in the case of the representation $\rho_X$ (see Remark \ref{rmk:equivalent_representations}), the definition of $\widetilde{\rho}_X$ is unique only up to equivalence.  Indeed, while the set $\widetilde{\B}_X$ is unequivocally defined, and algebra morphisms necessarily chosen forhe elements in $\widetilde{\B}_X$. Different orderings yield different representations,  however these representations agree up to inner conjugation in $S_m$. In other words, any two representations arising from different orderings of $\widetilde{\B}_X$ are equivalent.

\end{remark}

%%%%%%%%%%%%%%%%%%%%%%%%%%%%
%%%%%%%%%%% SECTION 3 %%%%%%%%%%%
%%%%%%%%%%%%%%%%%%%%%%%%%%%%

\section{Multiply-transitive groups}\label{sec:trans_groups}

We compile essential results on transitive finite groups that will be applied in the following sections. For a more in-depth treatment of this topic, the reader is encouraged to consult \cite[Chapter 1]{Cameron-permutations-groups} and \cite[Chapter 7]{Robinson}.

\begin{definition}\label{def:transitive}
Let $n$ and $k$ be positive integers. A permutation group $P\leq S_n$ is $k$-transitive if for any two ordered sets of $k$ distinct numbers $(x_1,\ldots,x_k)$ and $(y_1,\ldots,y_k)$, $1\leq x_i,y_i\leq n$, there exists $\sigma\in P$ such that $\sigma(x_i)=y_i$ for $i=1,\ldots,k$.

An abstract finite group $G$ is called $k$-transitive if there exists a permutation representation $\chi\colon G\to S_n$ such that the permutation group $\chi(G)\leq S_n$ is $k$-transitive. 
\end{definition}

\begin{remark}\label{rmk:transitive}

Observe that every finite abstract group $G$ is $1$-transitive, as both its left and right regular representations define $G$ as a $1$-transitive permutation subgroup of $S_{|G|}$ \cite[pp.\ 7--8]{Cameron-permutations-groups}. Consequently, it is natural to focus on groups that are multiply-transitive, that is, groups which are $k$-transitive for some $k \geq 2$. Moreover, note that if a finite abstract group $G$ is $k$-transitive, then it is also $l$-transitive for every $l \leq k$.

\end{remark}

From the Classification of Finite Simple Groups, the complete list of all possible multiply-transitive groups is known \cite[Theorem 4.11]{Cameron-permutations-groups}. Below, we recall the classification of $4$-transitive groups:

\begin{theorem}\label{thm:class_transitive_groups}
All the $k$-transitive groups, $k\geq 4$ are:
\begin{enumerate}[label={\rm (\roman{*})}] 
    \item The symmetric group $S_n$ is $n$-transitive.

    \item The alternating group $A_n\leq S_n$ is $(n-2)$-transitive.

    \item The Mathieu groups $M_n\leq S_n$, $n=11,12,23,24$, are $4$-transitive for $n=11,23$ and $5$-transitive for $n=12,24$.
\end{enumerate}    
\end{theorem}

\begin{lemma}\label{lem:unique_4_trans}
Let $ P \leq S_n $ be a $ 4 $-transitive permutation group. Then, for any faithful permutation representation  $ \psi\colon P \to S_n $, the image $ \psi(P) \leq S_n $ is also a $ 4 $-transitive permutation group. 
\end{lemma}

\begin{proof}
The result is straightforward for $ P = A_n $ and $ P = S_n $, as in these cases $ S_n $ contains only one conjugacy class of these subgroups. For the Mathieu groups, the claim follows from a case-by-case analysis of their possible permutation representations, as described in the \emph{ATLAS of Finite Groups Representations} \cite{Atlas-representation}. Specifically:
\begin{itemize}
    \item For $ M_n $ with $ n = 11, 23, 24 $, there exists a unique copy of $ M_n \leq S_n $ (up to inner conjugation within $ S_n $).
    \item For $ M_{12} $, there are two non-equivalent representations $ M_{12} \leq S_{12} $, both of which are $ 5 $-transitive.
\end{itemize}
\end{proof}

%%%%%%%%%%%%%%%%%%%%%%%%%%%%
%%%%%%%%%%% SECTION 4 %%%%%%%%%%%
%%%%%%%%%%%%%%%%%%%%%%%%%%%%

\section{Proof of Theorem \ref{thmx:main_A}}\label{sec:Proof_A}

To ease the readability of this section, we have structured it as a series of lemmas, all of which share the assumptions stated in Theorem \ref{thmx:main_A}. Throughout the discussion, we work over a ground field $ \K $ and consider an idempotent $ n $-dimensional evolution $ \K $-algebra $ X $ with a fixed natural basis $ \B = \{b_1, b_2, \ldots, b_n\} $. The structure matrix associated to this basis is denoted by $ M_{\B} = (\mu_{ij})_{{i,j}=1,\ldots,n} $.
 
As stated in Corollary \ref{cor:cor} and Remark \ref{rmk:equivalent_representations}, for any automorphism $ \varphi \in \Aut(X) $, there exist nonzero scalars $ \lambda_i \in \K $ for $ i = 1, \ldots, n $, and a permutation $ \sigma \in S_n $ such that $ \varphi(b_i) = \lambda_i b_{\sigma(i)} $ for all $ i = 1, \ldots, n $. Consequently, one can define the permutation representation $ \rho_X : \Aut(X) \to S_n $ by setting $ \sigma = \rho_X(\varphi) $.

Now, since $ \varphi \in \Aut(X)$ is an algebra morphism, the equality \( \varphi(b_i^2) = \varphi(b_i)^2 \) holds for each \( i = 1, \ldots, n \). Expanding both sides, we obtain:
\[
\varphi(b_i^2) = \varphi\left(\sum_j \mu_{ij} b_j\right) = \sum_j \mu_{ij} \lambda_j b_{\sigma(j)},
\]
\[
\varphi(b_i)^2 = (\lambda_i b_{\sigma(i)})^2 = \lambda_i^2 \sum_{\sigma(j)} \mu_{\sigma(i)\sigma(j)} b_{\sigma(j)}.
\]
Comparing the coefficients of \( b_{\sigma(j)} \), we deduce:
\begin{equation}\label{eq:funeq}
    \lambda_j \mu_{ij} = \lambda_i^2 \mu_{\sigma(i)\sigma(j)}, \quad \text{for all } i, j = 1, \ldots, n.
\end{equation}

Conversely, given a permutation \( \sigma \in S_n \) and coefficients \( \lambda_i \neq 0 \) for \( i = 1, \ldots, n \), one can construct a linear automorphism of \( X \) that maps each \( b_i \) to \( \lambda_i b_{\sigma(i)} \). This linear map becomes an algebra automorphism if and only if Equation \eqref{eq:funeq} is satisfied for all \( i, j = 1, \ldots, n \).

Now, in view of Equation \eqref{eq:funeq}, one infers that the way all possible $\sigma=\rho_X(\varphi)\in\rho_X\big(\Aut(X)\big)$ act on pairs $(i,j)$ for $i,j=1,\ldots, n$, i.e.\ the $2$-transitivity of the permutation group $\rho_X\big(\Aut(X)\big)$, should impose specific regularity conditions on the structure coefficients of $ X $. This is formalized in the following result.
\begin{lemma}\label{lemmacomplete}
  Let $\rho_X\big(\Aut(X)\big)\leq S_n$ be a  $2$-transitive group. Then, either $\mu_{ij}=0$ for all $i\neq j$, or $\mu_{ij}\neq 0$ for all $i\neq j$. Similarly if $\rho_X\big(\Aut(X)\big)$ is transitive, then either $\mu_{ii}=0$ for all $i \in \{1, \ldots, n\}$, or $\mu_{ii} \neq 0$ for all $i \in \{1, \ldots, n\}$.
\end{lemma}
\begin{proof}

Let $i \neq j$ and $k \neq l$ be two arbitrary pairs of indices. By $2$-transitivity, there exists $\varphi \in \Aut(X)$ such that $\sigma(i) = k$ and $\sigma(j) = l$, where $\sigma = \rho_X(\varphi)$. Using Equation \eqref{eq:funeq} for $i, j$, and noting that  $\lambda_i$ and $\lambda_j$ are nonzero, it follows that $\mu_{ij} = 0$ if and only if $\mu_{kl} = 0$.

Similarly, if $\rho_X\big(\Aut(X)\big)$ is transitive, for any $i$ and $j$, there exists $\phi \in \Aut(X)$ such that $\sigma (i) =j$ where $\sigma = \rho_X(\phi)$.  Applying Equation \eqref{eq:funeq} for $i=j$, we deduce that $\mu_{ii} = 0$ if and only if $\mu_{jj} = 0$.

\end{proof}

We can prove that $\rho_X$ is faithful if $\rho_X\big(\Aut(X)\big) \leq S_n$ is transitive, and if at least one diagonal structure coefficient, $\mu_{ii} = (M_{\B})_{ii}$, is nonzero for some $i \in \{1, \ldots, n\}$.

\begin{lemma}\label{lem:lemmacoef1}
Let  $\rho_X\big(\Aut(X)\big)\leq S_n$ be a transitive group, and suppose that there exists $k \in \{1,\ldots,n\}$ such that $\mu_{kk} \neq 0$. Then, for any $\varphi \in \Aut(X)$, the nonzero scalars $\lambda_i$ associated to $\varphi$ by Corollary \ref{cor:cor} are  uniquely determined by $\sigma=\rho_X(\varphi).$ Namely,  $$\lambda_i=\frac{\mu_{ii}}{\mu_{\sigma(i)\sigma(i)}},$$ 
for all $i \in \{1, \ldots, n\}$ and therefore $\rho_X$ is faithful.
\end{lemma}
\begin{proof}
By Lemma \ref{lemmacomplete}, if $\mu_{kk}\neq 0$, for some $k \in \{1,\ldots,n\}$, then $\mu_{ii} \neq 0$ for all $i \in \{1,\ldots,n\}$.  Then, by Equation \eqref{eq:funeq} with $i=j$ and $\sigma=\rho_X(\varphi)\in\rho_X\big(\Aut(X)\big)$,  we obtain that $\displaystyle \lambda_i^2=\frac{\mu_{ii}}{\mu_{\sigma(i)\sigma(i)}}\lambda_i$. Since $\lambda_i\neq 0$ we can simplify to $\displaystyle \lambda_i=\frac{\mu_{ii}}{\mu_{\sigma(i)\sigma(i)}}$.    
\end{proof}

Moreover, if  $\rho_X\big(\Aut(X)\big)\leq S_n$ is $2$-transitive, with $n \geq 3$, we have the following:

\begin{lemma}\label{lem:lemmacoef2}
Let $n \geq 3$, and let $J$ denote the set of triples of pairwise distinct indices $(i, j, k)$, where $i, j, k \in \{1, \ldots, n\}$. Suppose that $\rho_X\big(\Aut(X)\big) \leq S_n$ is a $2$-transitive group and that $\mu_{ij} \neq 0$ for some $i \neq j$. Then, for any $\varphi \in \Aut(X)$, the  nonzero scalars $\lambda_i$, associated to $\varphi$ by Corollary~\ref{cor:cor}, are uniquely determined by $\sigma = \rho_X(\varphi)$ and therefore $\rho_X$ is faithful.
More precisely, the function 
\[
\begin{array}{rcl}
J \times S_n & \xrightarrow{B} & \K^* \\
(i, j, k, \sigma) & \mapsto & \frac{\mu_{ik}}{\mu_{\sigma(i)\sigma(k)}}  \frac{\mu_{\sigma(j)\sigma(k)}}{\mu_{jk}} \frac{\mu_{ji}}{\mu_{\sigma(j)\sigma(i)}}
\end{array}
\]
is well-defined, and if $\varphi \in \Aut(X)$ with $\sigma = \rho_X(\varphi)$, then \(\lambda_i = B(i, j, k, \sigma)\) for any \((i, j, k) \in J\).
\end{lemma}

\begin{proof}
First, notice that the map $B$ is well defined since, according to Lemma \ref{lemmacomplete},  $\mu_{ij}\neq 0$ for every $i\neq j$.
Now, we can repeatedly apply Equation \eqref{eq:funeq}, alternating the roles of $i$, $j$, and $k$ (which are assumed to be distinct), and dividing by the corresponding structure constants, to obtain:
\begin{gather}
 \lambda_k=\frac{\mu_{\sigma(j)\sigma(k)}}{\mu_{jk}} \lambda_j^2,\label{1}\\
\lambda_j^2=\frac{\mu_{ji}}{\mu_{\sigma(j)\sigma(i)}} \lambda_i,\label{2}\\
\lambda_k=\frac{\mu_{\sigma(j)\sigma(k)}}{\mu_{jk}}\frac{\mu_{ji}}{\mu_{\sigma(j)\sigma(i)}} \lambda_i\text{ by \eqref{1} and \eqref{2},}\label{3}\\
\lambda_k=\frac{\mu_{\sigma(i)\sigma(k)}}{\mu_{ik}} \lambda_i^2, \label{4}\\
\frac{\mu_{\sigma(i)\sigma(k)}}{\mu_{ik}} \lambda_i^2=\frac{\mu_{\sigma(j)\sigma(k)}}{\mu_{jk}}\frac{\mu_{ji}}{\mu_{\sigma(j)\sigma(i)}} \lambda_i,\text{ by \eqref{3} and \eqref{4},}\label{5}
\end{gather}
hence, since $\lambda_i \neq 0$,  by \eqref{5}, we finally obtain $\displaystyle \lambda_i=\frac{\mu_{ik}}{\mu_{\sigma(i)\sigma(k)}}\frac{\mu_{\sigma(j)\sigma(k)}}{\mu_{jk}}\frac{\mu_{ji}}{\mu_{\sigma(j)\sigma(i)}}.$\end{proof}

\begin{remark}\label{rmk:injective_by_EL}

The injectivity of $\rho_X$, as established in both Lemmas \ref{lem:lemmacoef1} and \ref{lem:lemmacoef2}, can also be deduced using the arguments presented in \cite{Elduque-Labra-2019}. Specifically, \cite[Theorems 2.7 and 3.2]{Elduque-Labra-2019} identify the kernel of $\rho_X$ with the group of $(2^{\operatorname{b}(\Gamma(X,\B))}-1)$-roots of unity in $\K$, where $\Gamma(X,\B)$ denotes the directed graph associated with $X$ (see \cite[Section 2.3]{Cabrera-Siles-Velasco}), and $\operatorname{b}(\Gamma)$ denotes the balance of the directed graph $\Gamma$ (see \cite[Section 2]{Elduque-Labra-2019}). 
Now, if either $\mu_{ii}\ne 0$ for every $i \in \{1,\ldots, n\}$ (which follows from the hypothesis of Lemma \ref{lem:lemmacoef1}, using  Lemma \ref{lemmacomplete}) or $n\geq 3$ and $\mu_{ij}\ne 0$ for every pairwise different $i,j \in \{1,\ldots, n\}$ (which follows from the hypothesis of Lemma \ref{lem:lemmacoef2}), then ${\operatorname{b}(\Gamma(X,\B))}=1$ and the kernel of $\rho_X$ is trivial.
\end{remark}

\begin{remark}\label{rmk:injective_by_Nos}
Lemmas \ref{lem:lemmacoef1} and \ref{lem:lemmacoef2} not only establish the injectivity of $\rho_X$, but also provide explicit formulas for determining $\varphi \in \Aut(X)$ in terms of $\rho_X(\varphi)$ and the structure coefficients of $X$. This contrasts with the more indirect arguments presented in the previous remark. These formulas reveal that $\varphi$ can be determined locally by the interaction of three elements from the natural bases and their images, rather than requiring a global analysis of the entire algebra structure. Such a local approach proves to be useful in the proof of Theorem \ref{thmx:main_A}. \end{remark}

Now, we prove a slightly modified version of Theorem \ref{thmx:main_A}\ref{thmx:main_A.2}:

\begin{theorem}\label{thm:2-trans} Let $X$ be an idempotent $n$-dimensional evolution $\K$-algebra, and let the structure matrix $M_\B$ have a nonzero diagonal entry for some natural basis $\B$. If the group $\rho_X\big(\Aut(X)\big) \leq S_n$ is $2$-transitive, then $\rho_X\big(\Aut(X)\big) = S_n$. \end{theorem}

\begin{proof}

To prove that $\rho_X\big(\Aut(X)\big) = S_n$, it suffices to show that for any $\tau \in S_n$, there exists an automorphism $\varphi \in \Aut(X)$ such that $\rho_X(\varphi) = \tau$. As we observed at the beginning of this section, this is equivalent to finding nonzero scalars $\lambda_i^\tau$ that satisfy Equation \eqref{eq:funeq} for each pair $i,j$.
   
   Since by hypothesis  $(M_\B)_{ii} =\mu_{ii} \neq 0$, for some $i \in \{1, \ldots, n\}$,  if $\varphi \in \Aut(X)$ exists such that $\rho_X(\varphi) = \tau$, then 
 applying Lemma \ref{lem:lemmacoef1},  $\lambda_i^\tau = \frac{\mu_{ii}}{\mu_{\tau(i)\tau(i)}} \neq 0$ for all $i \in \{1, \ldots, n\}$. It then remains to verify that these nonzero scalars $\lambda_i^\tau$ satisfy Equation \eqref{eq:funeq} for each pair of distinct indices $i,j$.
 
 Let $i,j \in \{1, \ldots, n\}$ be distinct indices. Since $\rho_X\big(\Aut(X)\big) \leq S_n$ is $2$-transitive, there exists $\phi \in \Aut(X)$ such that $\rho_X(\phi) = \sigma \in S_n$, with $\sigma(i) = \tau(i)$ and $\sigma(j) = \tau(j)$. Now, let $\phi \in \Aut(X)$ be such that $\rho_X(\phi) = \sigma$ with coefficients $\lambda_k^\sigma$, $k  \in \{1, \ldots, n\}$. By Lemma \ref{lem:lemmacoef1}, we have that $\lambda_i^\tau = \lambda_i^\sigma$ and $\lambda_j^\tau = \lambda_j^\sigma$. Therefore, we obtain the following equality:
$$\lambda_j^\tau \mu_{ij} = \lambda_j^\sigma \mu_{ij} = (\lambda_i^\sigma)^2 \mu_{\sigma(i) \sigma(j)} = (\lambda_i^\tau)^2 \mu_{\tau(i) \tau(j)},$$
where the second equality follows from the fact that $\phi \in \Aut(X)$. Thus, the nonzero scalars $\lambda_i^\tau$ satisfy Equation \eqref{eq:funeq}, and they define an automorphism $\varphi \in \Aut(X)$ such that $\tau = \rho_X(\varphi) \in \rho_X\big(\Aut(X)\big)$.

 \end{proof}

We now have all the necessary ingredients to prove our main theorem:

\begin{proof}[Proof of Theorem \ref{thmx:main_A}]

Observe that Theorem \ref{thmx:main_A}\ref{thmx:main_A.1} is essentially a corollary of Lemmas \ref{lem:lemmacoef1} and \ref{lem:lemmacoef2}. We now turn our attention to proving Theorem \ref{thmx:main_A}\ref{thmx:main_A.2}.

Recall from Lemma \ref{lemmacomplete} that either $\mu_{ij} = 0$ for all $i \neq j$, or $\mu_{ij} \neq 0$ for all $i \neq j$. Similarly, either $\mu_{ii} = 0$ for all $i \in \{1, \ldots, n\}$, or $\mu_{ii} \neq 0$ for all $i$. If $\mu_{ii} \neq 0$ for every $i$, then Theorem \ref{thmx:main_A}\ref{thmx:main_A.2} follows directly from Theorem \ref{thm:2-trans}. Thus, we can assume that $\mu_{ij} \neq 0$ for $i \neq j$. Additionally, since $S_n$ has no proper $4$-transitive subgroups when $n < 5$, we further assume that $n \geq 5$.

We now proceed with an argument similar to that in the proof of Theorem \ref{thm:2-trans}: given $\tau \in S_n$, we construct $\varphi \in \Aut(X)$ such that $\rho_X(\varphi) = \tau$, by selecting nonzero scalars $\lambda_i^\tau$ that satisfy Equation \eqref{eq:funeq}, and defining $\varphi(b_i) = \lambda_i^\tau b_{\tau(i)}$.

By Lemma \ref{lem:lemmacoef2}, for $\varphi$ to be an automorphism, the scalars must be of the form $\lambda_i^\tau = B(i,j,k,\tau) \neq 0$ for some $i \neq j \neq k \neq i$. Moreover, they must be independent of the specific choices of $j$ and $k$: we select $j' \notin \{i,j,k\}$ and show that $B(i,j,k,\tau) = B(i,j',k,\tau)$. Indeed, since $\rho_X\big(\Aut(X)\big) \leq S_n$ is $4$-transitive, there exists $\phi \in \Aut(X)$ such that $\sigma = \rho_X(\phi)$ agrees with $\tau$ on the set $\{i,j,j',k\}$. Therefore, we have the following equalities: 
\begin{align*}
B(i,j,k,\tau) &= B(i,j,k,\sigma), \quad \text{because } \tau(i,j,k) = \sigma(i,j,k), \\
B(i,j,k,\sigma) &= B(i,j',k,\sigma), \quad \text{because $\sigma$ is induced by an automorphism of $X$}, \\
B(i,j',k,\sigma) &= B(i,j',k,\tau), \quad \text{because } \tau(i,j',k) = \sigma(i,j',k).
\end{align*}
By combining all three equations, we obtain that $B(i,j,k,\tau) = B(i,j',k,\tau)$. The argument showing that the definition of the scalars $\lambda_i^\tau$ is independent of $k$ follows similarly.

To complete the proof, it only remains to check Equation \eqref{eq:funeq} for the scalars $\lambda_i^\tau$ to obtain that $\varphi \in \Aut(X)$. Given $i, j \in \{1, \ldots, n\}$ (which may be equal), we select two additional auxiliary indices $k \neq l \in \{1, \ldots, n\}$, distinct from $i$ and $j$. We then proceed  as in the previous paragraph: we choose $\phi \in \Aut(X)$ such that $\sigma = \rho_X(\phi)$ agrees with $\tau$ when restricted to the set $\{i,j,k,l\}$. Let $\lambda_i^\sigma$ denote the nonzero scalars associated to $\phi$ by Corollary \ref{cor:cor}, then we have:
\begin{align*}
    \lambda_j^\tau \mu_{ij} &= B(j,k,l,\tau) \mu_{ij} = B(j,k,l,\sigma) \mu_{ij} = \lambda_j^\sigma \mu_{ij}, \; \text{because } \tau(j,k,l) = \sigma(j,k,l), \\
    \lambda_j^\sigma \mu_{ij} &= (\lambda_i^\sigma)^2 \mu_{\sigma(i) \sigma(j)}, \; \text{because } \sigma = \rho_X(\phi), \\
    (\lambda_i^\sigma)^2 \mu_{\sigma(i) \sigma(j)} &= B(i,k,l,\sigma)^2 \mu_{\sigma(i) \sigma(j)} \\
    & = B(i,k,l,\tau)^2 \mu_{\tau(i) \tau(j)} = (\lambda_i^\tau)^2 \mu_{\tau(i) \tau(j)}, \;\text{because } \tau(j,k,l) = \sigma(j,k,l).
\end{align*}

By combining these equations, we obtain
$$
\lambda_j^\tau \mu_{ij} = (\lambda_i^\tau)^2 \mu_{\tau(i) \tau(j)},
$$
which is precisely Equation \eqref{eq:funeq} for $i, j$, and $\tau$. Repeating this argument for all pairs $i, j$, we conclude that $\varphi \in \Aut(X)$, thus completing the proof.

\end{proof}

We finish this section proving Corollary \ref{corx:trans}:

\begin{proof}[Proof of Corollary \ref{corx:trans}]
Let $G \leq S_n$ be a $k$-transitive group with $k \geq 4$, and assume $\Aut(X) \cong G$. Recall that the kernel of the permutation representation $\Aut(X) \xrightarrow{\rho_X} S_n$ is a normal abelian subgroup \cite[Theorem 3.2]{Elduque-Labra-2019}. However, by the classification in Theorem \ref{thm:class_transitive_groups}, $G$ has no proper abelian normal subgroups. Hence, $\rho_X$ must be injective. Therefore, $\rho_X(\Aut(X)) \leq S_n$ and by Lemma \ref{lem:unique_4_trans}, it follows that $\rho_X(\Aut(X)) \leq S_n$ is $k$-transitive. Applying Theorem \ref{thmx:main_A}, we conclude that $\Aut(X) \cong S_n$.
\end{proof}

\begin{remark}\label{rmk:negative_Q1}
    Corollary \ref{corx:trans} shows that there is no $n$-dimensional idempotent evolution $\K$-algebra $X$, with $\K$ any field, such that $\Aut(X) \cong A_n$ for $n \geq 6$ or $\Aut(X) \cong M_n$ for $n = 11, 12, 23, 14$. In particular, the standard permutation representation $A_n \hookrightarrow S_n$ for $n \geq 6$ and $M_n \hookrightarrow S_n$ for $n = 11, 12, 23, 14$ cannot be realized as $\rho_X$ for any $X$ over any ground field $\K$. This illustrates that the answer to Question \ref{Quest:basica} is negative for some representations.
\end{remark}

%%%%%%%%%%%%%%%%%%%%%%%%%%%%
%%%%%%%%%%% SECTION 5 %%%%%%%%%%%
%%%%%%%%%%%%%%%%%%%%%%%%%%%%

 \section{Actions on idempotents}\label{sec:thmx_C}
 
 We have just observed (see Remark \ref{rmk:negative_Q1}) that certain permutation representations cannot be realized as $\rho_X$ for any idempotent evolution algebra $X$, indicating that the answer to Question \ref{Quest:basica} is negative in some cases. However, in this section, we show that the scenario changes completely when considering the permutation representation $\widetilde{\rho}_X$. Specifically, we establish a positive answer to Question \ref{Quest:basica2} by proving Theorem \ref{thmx:main_C}.

To that purpose, we build upon the ideas in \cite{CLTV} and construct idempotent evolution algebras from finite simple graphs. We begin by recalling \cite[Theorem~2.1]{CGV-PermMod}, a refinement of \cite[Theorem 3.1]{CMV}, which shows that permutation representations can be realized in the context of graphs:

 \begin{theorem}\cite[Theorem~2.1]{CGV-PermMod}\label{thm:CGV}
 Let $G$ be a finite group, $V$ be a finite non-empty set and $\xi:G \rightarrow \operatorname{Sym}(V)$ a permutation representation of $G$ on $V$. Then, there are infinitely many non-isomorphic (simple, undirected) finite connected graphs $\mathcal{G}$ such that:
 \begin{enumerate}[label={\rm (\roman{*})}]
     \item $V\subset V(\mathcal{G})$ and $V$ is $\Aut(\G)$-invariant.
     \item Every vertex in $\G$ has degree at least 2.
     \item $\Aut(\G)\cong G$;
     \item the restriction $G\cong \Aut(\G)\rightarrow \operatorname{Sym}(V)$ is precisely $\xi$.
 \end{enumerate}
 \label{thm:graphs}
 \end{theorem}

To clarify how we use Theorem \ref{thm:CGV}, we restate Theorem \ref{thmx:main_C} as follows:

\begin{theorem}\label{thm:repres}
Let $\K$ be a field, $G$ a finite group, $V$ a finite non-empty set, and $\xi: G \to \operatorname{Sym}(V)$ a permutation representation of $G$ on $V$. There exists an idempotent evolution $\K$-algebra $X$ with a natural basis $\B$ such that:

\begin{enumerate}[label={\rm (\arabic{*})}]
    \item The elements of $\widetilde{B}_X$, the set of natural idempotent elements of $X$, are indexed by elements of $V,$ hence we may identify $\widetilde{B}_X \cong V$. Recall that $\widetilde{B}_X$ is $\Aut(X)$-invariant.
    \item $\Aut(X) \cong G$.
    \item The permutation representation $\widetilde{\rho}_X: G \cong \Aut(X) \to \operatorname{Sym}(\widetilde{B}_X) \cong \operatorname{Sym}(V)$ is precisely $\xi$.
\end{enumerate}
 \end{theorem}

 \begin{proof}
 We apply Theorem \ref{thm:graphs} to first obtain a graph $\G$ from which we can explicitly derive an evolution algebra with the desired properties, following a procedure similar to the one outlined in \cite{CLTV}.
 
 Let $\G$ be the given graph, with vertices labeled $V(\G)=\{ v_1,\ldots,v_n\}$ in such a way that $V=\{ v_1,\ldots,v_r\} \subset V(\G)$. We now define the algebra $X$ by giving a natural basis $\B = \{b_{v_i}:v_i\in V(\G)\} \cup \{b_{e_j}:e_j \in E(\G)\}$ and the squares of its elements:
  \begin{itemize}
    \item $\displaystyle b_{v_i}^2 = b_{v_i}$, for $v_i \in V$;
     \item $\displaystyle b_{v_i}^2 = b_{v_i} + \sum_{w \in V} b_w$, for $v_i \in V(\G) \setminus V$;
       \item $\displaystyle b_{e_i}^2 = b_{e_i} + b_{v_a} + b_{v_b}$, for $e_i \in E(\G)$ and $e_i = \{v_a, v_b\}$.
\end{itemize}
With this definition, we identify a vertex $v \in $ with $b_v$, so that $V \equiv \widetilde{\B}_X = { b \in \B : b^2 = b }$ as shown in (1). Next we analyze the group $\Aut(X)$ to check that $(2)$ and $(3)$ are also satisfied. 
First, note that by ordering $\B$ as defined earlier, the coefficient matrix $M_\B$ becomes upper-triangular with ones along the diagonal. This ensures that $M_\B$ is invertible, implying that $X$ is an idempotent evolution algebra. Then, by Corollary \ref{cor:cor},  we know that any automorphism of $X$ permutes the elements of $\B$ up to multiplication by nonzero scalars.
 
Since the number of nonzero coefficients, in each row of the coefficients matrix, is invariant under automorphisms, in the case where $r \neq 2$, we deduce that for any  $\varphi \in \Aut(X)$, the following holds: for any $v \in V$, $\varphi(b_v) \in \K^* b_w$ for some $w \in V$; for any $u \in V(\G) \setminus V$, $\varphi(b_u) \in \K^* b_w$ for some $w \in V(\G) \setminus V$; and for any $e \in E(\G)$, $\varphi(b_e) \in \K^* b_{e'}$ for some $e' \in E(\G)$.

  In the case where \(r = 2\), this argument fails, as it could be that for some \(v \in V(\G) \setminus V\),  \(\varphi(b_v) \in \K^* b_e\) with \(e = \{u, w\}\). However, we can rule this out as follows. Since each vertex of \(\G\) has degree at least 2, there exists an edge \(f = \{v, v'\}\). Then, we have
$$
\varphi(b_f)^2 = \varphi(b_f^2) = \varphi(b_f + b_v + b_{v'}) = \varphi(b_f)  + \varphi(b_v)  + \varphi(b_{v'})= \varphi(b_f)+ \lambda b_e + \varphi(b_{v'}) 
$$
for some \(\lambda \neq 0\). By Corollary \ref{cor:cor}, \(\varphi(b_f)\) is a scalar multiple of a basis element. According to the structure constants, the square of this basis element can only have a nonzero \(b_e\) component if \(\varphi(b_f) = \mu b_e\) for some \(\mu \neq 0\). This is a contradiction, as it would imply that \(\varphi\) is not injective. Similarly, \(\varphi(b_{v'})\) cannot contribute to the \(b_e\) component. Therefore, even in the case \(r = 2\), the same conclusion holds as in the previous paragraph.

 We deduce that there exists a permutation \(\sigma \in \operatorname{Sym}(V(\G))\) that preserves $V$  such that \(\varphi(b_v) \in \K^* b_{\sigma(v)}\) for vertices,  and a permutation of the edges that we also denote $\sigma$ such that $\varphi(b_e) \in \K^* b_{\sigma(e)}$ for edges. For this automorphism $\varphi$, we obtain scalars $\lambda_v, \lambda_e \in \K^*$ for each $v \in V(\G)$ and $e \in E(\G)$, such that $\varphi(b_v) = \lambda_v b_{\sigma(v)}$ and $\varphi(b_e) = \lambda_e b_{\sigma(e)}$.

Now, for $v \in V$, we have:
\[
\lambda_v b_{\sigma(v)} = \varphi(b_v) = \varphi(b_v^2) = \varphi(b_v)^2 = \lambda_v^2 b_{\sigma(v)}^2 = \lambda_v^2 b_{\sigma(v)}.
\]
Thus, $\lambda_v = 1$.

Next, for $v \in V(\G) \setminus V$, we consider:
\[
\varphi(b_v^2) = \varphi(b_v + \sum_{w \in V} b_w) = \lambda_v b_{\sigma(v)} + \sum_{w \in V} \lambda_w b_{\sigma(w)} = \lambda_v b_{\sigma(v)} + \sum_{w \in V} b_w,
\]
and
\[
\varphi(b_v)^2 = (\lambda_v b_v)^2 = \lambda_v^2 \left( b_{\sigma(v)} + \sum_{w \in V} b_w \right).
\]
Thus, $\lambda_v = 1$ for all $v \in V(\G)$.

 Now, let $e \in E(\G)$, with $e = \{v, w\}$ and $\sigma(e) = \{v', w'\}$. We have
$$\varphi(b_e^2) = \varphi(b_e + b_v + b_w) = \lambda_e b_{\sigma(e)} + \lambda_v b_{\sigma(v)} + \lambda_w b_{\sigma(w)}, $$
and
$$\varphi(b_e)^2 = (\lambda_e b_e)^2 = \lambda_e^2 (b_{\sigma(e)} + b_{v'} + b_{w'}).$$
Therefore, $\lambda_e = 1$ and $\{v', w'\} = \{\sigma(v), \sigma(w)\}$. This shows that the permutation $\sigma$ on the vertices is a graph isomorphism, while the permutation on the edges is the induced one.
With this, we have proved that an automorphism of $X$ induces an automorphism of $\G$. 

The converse is also true: consider an automorphism of $\G$, which can be described by a permutation $\sigma$ of the vertices, inducing a permutation of the edges (denoted by the same name) such that $\sigma(\{v,w\}) = \{\sigma(v), \sigma(w)\}$.
We define an automorphism $\varphi \in \Aut(X)$ on the natural basis by setting $\varphi(b_v) = b_{\sigma(v)}$ for $v \in V(\G)$ and $\varphi(b_e) = b_{\sigma(e)}$ for $e \in E(\G)$. To verify that this is an automorphism, we check that it preserves the squares of the basis elements. Indeed, 
for $v \in V$, we have that $\sigma(v) \in V$, so $\varphi(b_v^2) = \varphi(b_v) = b_{\sigma(v)} = (b_{\sigma(v)})^2$. Now, for 
 $v \in V(\G) \setminus V$, we have that
$\varphi(b_v^2) = \varphi(b_v + \sum_{w \in V} b_w) = b_{\sigma(v)} + \sum_{w \in V} b_{\sigma(w)} = b_{\sigma(v)} + \sum_{w \in V} b_w = (b_{\sigma(v)})^2$ using the fact that $\sigma$ permutes $V$.
Finally, for $e = \{u,v\} \in E(\G)$, we have that $\varphi(b_e^2) = \varphi(b_e + b_u + b_v) = b_{\sigma(e)} + b_{\sigma(u)} + b_{\sigma(v)} = (b_{\sigma(e)})^2$, since $\sigma$ is a graph automorphism.
Thus, $\varphi \in \Aut(X)$.

The maps we have constructed between $\Aut(X)$ and $\Aut(\G)$ are clearly mutual inverses and respect composition. Therefore, we conclude that $\Aut(X) \cong \Aut(\G)$ as groups. Moreover, the restriction of the action to $V$ commutes with this isomorphism, thus completing the proof of the theorem.

 \end{proof}

 %%%%%%%%%% BIBLIOGRAFIA EN BIBTEX %%%%%%%%%%%%%%

\bibliographystyle{abbrv}
\bibliography{references}

\end{document}